\documentclass[10pt]{amsart}
\usepackage{amsmath, amssymb}
 \usepackage{mathrsfs}
\newcommand{\no}[1]{#1}
\renewcommand{\no}[1]{}
\no{\usepackage{times}\usepackage[subscriptcorrection, slantedGreek, nofontinfo]{mtpro}
\renewcommand{\Delta}{\upDelta}}
\usepackage{color}

\date{\today}
\newtheorem{theorem}{Theorem}[section]
\newtheorem{proposition}{Proposition}[section]
\newtheorem{lemma}{Lemma}[section]

\newtheorem{corollary}{Corollary}[section]

\theoremstyle{remark}

\numberwithin{equation}{section}

\title[Logarithmic stability in determining  a boundary coefficient]{Logarithmic stability in determining a boundary coefficient in an IBVP for the wave equation}

\author[Ka\"{\i}s Ammari]{Ka\"{\i}s Ammari}
\address{UR Analysis and Control of PDEs, UR 13ES64, Department of Mathematics, Faculty of Sciences of Monastir, University of Monastir, 5019 Monastir, Tunisia }
\email{kais.ammari@fsm.rnu.tn}

\author[Mourad Choulli]{Mourad Choulli}
\address{Institut \'Elie Cartan de Lorraine, UMR CNRS 7502, Universit\'e de Lorraine, Boulevard des Aiguillettes, BP 70239, 54506 Vandoeuvre les Nancy cedex - Ile du Saulcy, 57045 Metz cedex 01, France}
\email{mourad.choulli@univ-lorraine.fr}

\date{}

\begin{document}

\begin{abstract}
In \cite{AC1} we introduced a method combining together an observability inequality and a spectral decomposition to get a logarithmic stability estimate for the inverse problem of determining both the potential and the damping coefficient in a dissipative wave equation from boundary measurements. The present work deals with an adaptation of that method to obtain a logarithmic stability estimate for the inverse problem of determining a boundary damping coefficient from boundary measurements. As in our preceding work, the different boundary measurements are generated by varying one of the initial conditions.


\end{abstract}

\maketitle

\tableofcontents


\section{Introduction}
We are concerned with an inverse problem for the wave equation when the spatial domain is the square $\Omega=(0,1)\times (0,1)$. To this end we consider the following initial-boundary value problem (abbreviated to IBVP in the sequel) :
\begin{equation}\label{1.1}
\left\{
\begin{array}{lll}
 \partial _t^2 u - \Delta u = 0 \;\; &\mbox{in}\;   Q=\Omega \times (0,\tau), 
 \\
u = 0 &\mbox{on}\;  \Sigma _0=\Gamma _0 \times (0,\tau), 
\\
\partial _\nu u +a\partial _tu=0 &\mbox{on}\;  \Sigma _1=\Gamma _1 \times (0,\tau), 
\\
u(\cdot ,0) = u^0,\; \partial_t u (\cdot ,0) = u^1.
\end{array}
\right.
\end{equation}
Here 
\begin{align*}
&\Gamma _0=((0,1)\times \{1\})\cup (\{1\}\times (0,1)),
\\
&\Gamma _1=((0,1)\times \{0\})\cup (\{0\}\times (0,1))
\end{align*}
and $\partial _\nu =\nu \cdot \nabla$ is the derivative along $\nu$, the unit normal vector pointing outward of $\Omega$. We note that $\nu$ is everywhere defined except at the vertices of $\Omega$ and we denote by $\Gamma = \Gamma_0 \cup \Gamma_1$. The boundary coefficient $a$ is usually called the boundary damping coefficient.

\smallskip
In the rest of this text we identify $a_{|(0,1)\times \{0\}}$ by $a_1=a_1(x)$, $x\in (0,1)$ and $a_{|\{0\}\times (0,1)}$ by $a_2=a_2(y)$, $y\in (0,1)$. In that case it is natural to identify $a$, defined on $\Gamma _1$, by the pair $(a_1,a_2)$.
 
\subsection{The IBVP} 

We fix $1/2<\alpha \le 1$ and we assume that $a\in \mathscr{A}$, where
\[
\mathscr{A}=\{ b=(b_1,b_2)\in C^\alpha ([0,1])^2,\; b_1(0)=b_2(0),\; b_j\geq 0\}. 
\]
This assumption guarantees that the multiplication operator by $a_j$, $j=1,2$, defines a bounded operator on $H^{1/2}((0,1))$. The proof of this fact will be proved in Appendix A.

\smallskip
Let $V=\{ u\in H^1(\Omega );\; u=0\; \textrm{on}\; \Gamma _0\}$ and we consider on $V\times L^2(\Omega )$ the linear unbounded operator $A$ given by
\[
A_a= (w,\Delta v),\quad D(A_a)=\{ (v,w)\in V\times V;\; \Delta v\in L^2(\Omega )\; \textrm{and}\; \partial _\nu v=-aw\; \textrm{on}\; \Gamma _1\}.
\]

One can prove that $A_a$ is a m-dissipative operator on the Hilbert space $V\times L^2(\Omega )$ (for the reader's convenience we detail the proof in Appendix B). Therefore, $A_a$ is the generator of a strongly continuous semigroup of contractions $e^{tA_a}$. Hence, for each $(u^0,u^1)$, the IBVP \eqref{1.1} possesses a unique solution denoted by $u_a=u_a(u^0,u^1)$ so that
\[
(u_a,\partial _tu_a)\in C([0,\infty );D(A_a))\cap C^1([0,\infty ),V\times L^2(\Omega )).
\]

\subsection{Main result}

For $0<m\leq M$, we set
\[
\mathscr{A}_{m,M}=\{ b=(b_1,b_2)\in \mathscr{A}\cap H^1(0,1)^2;\; m\leq b_j,\; \|b_j\|_{H^1(0,1)}^2\leq M\}.
\]

Let $\mathcal{U}_0$ given by 
\[
\mathcal{U}_0=\{ v\in V;\; \Delta v\in L^2(\Omega )\; \textrm{and}\; \partial _\nu v=0\; \textrm{on}\; \Gamma _1\}.
\]
We observe that $\mathcal{U}_0\times \{0\}\subset D(A_a)$, for any $a\in \mathscr{A}$.

\smallskip
Let $C_a\in \mathscr{B}(D(A_a);L^2(\Sigma _1))$ defined by
\[
C_a(u^0,u^1)=\partial _\nu u_a(u^0,u^1){_{|\Gamma _1}}.
\]
We define the initial to boundary operator
\[
\Lambda _a :u^0\in \mathcal{U}_0 \longrightarrow C_a(u^0,0)\in L^2(\Sigma _1).
\]
Clearly $C_a\in \mathscr{B}(D(A_a);L^2(\Sigma _1))$ implies that $\Lambda _a\in \mathscr{B}(\mathcal{U}_0;L^2(\Sigma _1))$, when $\mathcal{U}_0$ is identified  to a subspace of $D(A_a)$ endowed with the graph norm of $A_a$. Precisely the norm in $\mathcal{U}_0$ is the following one
\[
\|u^0\|_{\mathcal{U}_0}=\left(\|u^0\|_V^2+\|\Delta u^0\|_{L^2(\Omega )}^2\right)^{1/2}.
\]
Henceforth, for simplicity sake, the norm of $\Lambda_a-\Lambda _0$ in $\mathscr{B}(\mathcal{U}_0;L^2(\Sigma _1))$ will denoted by $\| \Lambda_a-\Lambda _0\|$.

\begin{theorem}\label{theorem1.1}
There exists $\tau _0>0$ so that for any $\tau >\tau _0$, we find a constant $c>0$ depending only on $\tau$ such that
\begin{equation}\label{1.2}
 \|a-0\|_{L^2((0,1))^2} \le cM\left(\left| \ln \left(m^{-1}\| \Lambda_a-\Lambda _0\|\right)\right|^{-1/2}+m^{-1}\|\Lambda_a-\Lambda _0\|_{L^2(\Sigma _1)}\right),
\end{equation}
for each $a\in \mathcal{A}_{m,M}$.
\end{theorem}

We point out that our choice of the domain $\Omega$ is motivated by the fact the spectral analysis of the laplacian under mixed boundary condition is very simple in that case. However this choice has the inconvenient that the square domain $\Omega$ is no longer smooth. So we need to prove an observability inequality associated to this non smooth domain. This is done by adapting the existing results. We note that the key point in establishing this observability inequality relies on a Rellich type identity for the domain $\Omega$. 

\smallskip
The inverse problem we discuss in the present paper remains largely open for an arbitrary (smooth) domain as well as for the stability around a non zero damping coefficient. Uniqueness and directional Lipschitz stability, around the origin, was established by the authors in \cite{AC2}.

\smallskip
The determination of a potential and/or the sound speed coefficient in a wave equation from the so-called Dirichlet-to-Neumann map was extensively studied these last decades. We refer to the comments in \cite{AC1} for more details.


\section{Preliminaries}

\subsection{Extension lemma}

We decompose $\Gamma _1$ as follows $\Gamma _1=\Gamma _{1,1}\cup \Gamma_{1,2}$, where $\Gamma _{1,1}=(0,1)\times \{0\}$ and $\Gamma_{1,2}=\{0\}\times (0,1)$. Similarly, we write $\Gamma _0=\Gamma_{0,1}\cup\Gamma _{0,2}$, with $\Gamma_{0,1}=\{1\}\times (0,1)$ and $\Gamma_{0,2}=(0,1)\times \{1\}$.

\smallskip

Let $(g_1,g_2)\in L^2((0,1))^2$. We  say that the pair $(g_1,g_2)$ satisfies the compatibility condition of the first order at the vertex $(0,0)$ if
\begin{equation}\label{1.3}
\int_0^1|g_1(t)-g_2(t)|^2\frac{dt}{t}<\infty .
\end{equation}

Similarly, we can define the compatibility condition of the  first order at the other vertices of $\Omega$.

\smallskip
We need also to introduce compatibility conditions of the second order. Let $(f_j,g_j)\in H^1((0,1))\times L^2((0,1))$, $j=1,2$. We say that the pair $[(f_1,g_1),(f_2,g_2)]$ satisfies the compatibility conditions of second order at the vertex $(0,0)$ when
\begin{equation}
f_1(0)=f_2(0), \quad \int_0^1|f'_1(t)-g_2(t)|^2\frac{dt}{t}<\infty \quad \textrm{and}\quad  \int_0^1|g_1(t)-f'_2(t)|^2\frac{dt}{t}<\infty .
\end{equation}

The compatibility conditions of the second order at the other vertices of $\Omega$ are defined in the same manner.

\smallskip
The following theorem is a special case of \cite[Theorem 1.5.2.8, page 50]{Gr}.

\begin{theorem}\label{theorem1.3}
(1) The mapping 
\[
w\longrightarrow (w_{|\Gamma _{0,1}},w_{|\Gamma _{0,2}},w_{|\Gamma _{1,1}},w_{|\Gamma _{1,2}})=(g_1,\ldots ,g_4),
\]
defined on $\mathscr{D}(\overline{\Omega} )$ is extended from $H^1(\Omega )$ onto the subspace of $H^{1/2}((0,1))^4$ consisting in functions $(g_1,\ldots ,g_4)$ so that the compatibility condition of the first order  is satisfied at each vertex of $\Omega$ in a natural way with the pairs $(g_j,g_k)$. 
\\
(2) The mapping
\[
w\rightarrow (w_{|\Gamma _{0,1}}, \partial _x w_{|\Gamma _{0,1}}, w_{|\Gamma _{0,2}},\partial _y w_{|\Gamma _{0,2}} w_{|\Gamma _{1,1}},-\partial _y w_{|\Gamma _{1,1}}, w_{|\Gamma _{1,2}},-\partial _xw_{|\Gamma _{1,2}})=((f_1,g_1),\ldots (f_4,g_4))
\]
defined on $\mathscr{D}(\overline{\Omega} )$ is extended from $H^2(\Omega )$ onto the subspace of $[H^{3/2}((0,1))\times H^{1/2}((0,1))]^4$ of functions $((f_1,g_1),\ldots (f_4,g_4))$ so that the compatibility conditions of the second order are satisfied at each vertex of $\Omega$ in a natural way with the pairs $[(f_j,g_j),(f_k,g_k)]$.
\end{theorem}

\begin{lemma}\label{lemma1.2}
(Extension lemma) Let $g_j\in H^{1/2}((0,1))$, $j=1,2$, so that $(g_1,g_2)$, $(g_1,0)$ and $(g_2,0)$ satisfy the first order compatibility condition respectively at the  vertices $(0,0)$, $(1,0)$ and $(0,1)$. Then there exists $u\in H^2(\Omega )$ so that $u=0$ on $\Gamma _0$ and $\partial _\nu u=g_j$ on $\Gamma _{1,j}$, $j=1,2$.
\end{lemma}

\begin{proof}
(i) We define  $f_1(t)=\int_0^tg_2(s)ds $ and $f_2(t)=\int_0^tg_1(s)ds$. Then $(f_1,g_1)$ and  $(f_2,g_2)$ satisfy the compatibility conditions of the second order at the vertex $(0,0)$.

\smallskip
(ii) Let $\widetilde{g}_1\in H^{1/2}((0,1))$ be such that $\int_0^1\frac{|\widetilde{g}_1(t)|^2}{t}dt<\infty$.  Let $\widetilde{f}_1(t)=\int_0^tg_2(s)ds$. Hence, it is straightforward to check that $(\widetilde{f}_1,\widetilde{g}_1)$ and $(0,g_2)$ satisfy the  compatibility conditions of the second order at $(0,0)$.

\smallskip
(iii) From steps (i) and (ii) we derive that the pairs $[(f_1,g_1),(f_2,g_2)]$, $[(f_1,g_1),(0,g_2)]$ and $[(0,g_1),(f_2,g_2)]$ satisfy the second order compatibility conditions respectively at the  vertices $(0,0)$, $(1,0)$ and $(0,1)$. We see that unfortunately the pair $[(0,g_1),(0,g_2)]$  doesn't satisfy necessarily the compatibility conditions of the second order at the vertex $(1,1)$. 
We pick $\chi \in C^\infty (\mathbb{R})$ so that $\chi =1$ in a neighborhood of $0$ and $\chi =0$ in a neighborhood of $1$. Then
$[(0,\chi g_1),(0,\chi g_2)]$ satisfies the compatibility condition of the second order at the vertex $(1,1)$. Since this construction is of local character at each vertex, the cutoff function at the vertex $(1,1)$ doesn't  modify the construction at the other vertices. In other words, the compatibility conditions of the second order are preserved at the other vertices. We complete the proof by applying Theorem \ref{theorem1.3}.
\end{proof}

\begin{corollary}\label{corollary2.1}
Let $a=(a_1,a_2)\in \mathscr{A}$ and $g_j\in H^{1/2}((0,1))$, $j=1,2$, so that $(g_1,g_2)$, $(g_1,0)$ and $(g_2,0)$ satisfy the first order compatibility condition respectively at the  vertices $(0,0)$, $(1,0)$ and $(0,1)$. Then there exists $u\in H^2(\Omega )$ so that $u=0$ on $\Gamma _0$ and $\partial _\nu u=a_jg_j$ on $\Gamma _{1,j}$, $j=1,2$.
\end{corollary}

\begin{proof}
It is sufficient to prove that $(a_1g_1, a_2,g_2)$ and $(a_jg_j, 0)$, $j=1,2$, satisfy the first order compatibility condition at $(0,0)$ with $a_1(0)=a_2(0)$ for the first pair and without any condition on $a_j$ for the second pair.

\smallskip
Using $a_1(0)=a_2(0)$, we get
\begin{align*}
t^{-1}|a_1(t)-a_2(t)|^2 &\le 2t^{-1}|a_1(t)-a_1(0)|^2+2t^{-1}|a_2(t)-a_2(0)|^2
\\ &\le 2t^{-1+2\alpha}([a_1]_\alpha^2 +[a_2]_\alpha^2 )
\\ &\le 2([a_1]_\alpha^2 +[a_2]_\alpha ^2).
\end{align*}
This estimate together with the following one
\[
|a_1(t)g_1(t)-a_2(t)g_2(t)|^2\le 2|a_1(t)-a_2(t)|^2|g_1(t)|^2+2|a_2(t)|^2|g_1(t)-g_2(t)|^2
\]
yield
\[
\int_0^1|a_1(t)g_1(t)-a_2(t)g_2(t)|^2\frac{dt}{t}\le 4([a_1]_\alpha^2 +[a_2]_\alpha^2 )\|f\|_{L^2((0,1))}+2\|a_2\|_{L^\infty ((0,1))}\int_0^1|g_1(t)-g_2(t)|^2\frac{dt}{t}.
\]
Hence
\[
\int_0^1|g_1(t)-g_2(t)|^2\frac{dt}{t} <\infty \Longrightarrow \int_0^1|a_1(t)g_1(t)-a_2(t)g_2(t)|^2\frac{dt}{t}<\infty .
\]
If $(g_j,0)$ satisfies the first compatibility at the vertex $(0,0)$. Then 
\[
\int_0^1|g_j(t)|^2\frac{dt}{t}<\infty.
\]
Therefore
\[
\int_0^1|a_jg_j(t)|^2\frac{dt}{t} \le \|a_j\|_{L^\infty ((0,1))}^2\int_0^1|g_j(t)|^2\frac{dt}{t} <\infty .
\]
Thus $(a_jg_j,0)$ satisfies also the first compatibility at the vertex $(0,0)$.
\end{proof}

\subsection{Observability inequality}

We discuss briefly how we can adapt the existing results to get an observability inequality corresponding to our IBVP. We first note that
\begin{align*}
&\Gamma _0\subset \{ x\in \Gamma ;\; m(x)\cdot \nu (x)<0\},
\\
&\Gamma _1\subset \{ x\in \Gamma ;\; m(x)\cdot \nu (x)>0\},
\end{align*}
where $m(x)=x-x_0$, $x\in \mathbb{R}^2$, and $x_0=(\alpha ,\alpha )$ with $\alpha >1$.

\smallskip
The following Rellich identity is a particular case of identity \cite[(3.5), page 227]{Gr2}: for each $3/2<s<2$ and $\varphi \in H^s(\Omega )$ satisfying $\Delta \varphi \in L^2(\Omega )$, 
\begin{equation}\label{1.5}
2\int_\Omega \Delta \varphi (m\cdot \nabla \varphi )dx= 2\int_\Gamma \partial _\nu \varphi (m\cdot \nabla \varphi )d\sigma -\int_\Gamma (m\cdot \nu ) |\nabla \varphi |^2d\sigma .
\end{equation}

\begin{lemma}\label{lemma1.1}
Let $(v,w)\in D(A_a)$. Then
\[
2\int_\Omega \Delta v (m\cdot \nabla v )dx= 2\int_\Gamma \partial _\nu v (m\cdot \nabla v)d\sigma -\int_\Gamma (m\cdot \nu ) |\nabla v|^2d\sigma .
\]
\end{lemma}

\begin{proof}
Let $(v,w)\in D(A_a)$. By Corollary \ref{corollary2.1}, there exists $\widetilde{v}\in H^2(\Omega )$ so that $\widetilde{v}=0$ on $\Gamma_0$ and  $\partial _\nu \widetilde{v}=-aw$ on $\Gamma _1$. In light of the fact that $z=v-\widetilde{v}$ is such that $\Delta z\in L^2(\Omega )$, $z=0$ on $\Gamma _0$ and $\partial _\nu z=0$ on $\Gamma _1$, we get $z\in H^s(\Omega )$ for some $3/2<s<2$ by \cite[Theorem 5.2, page 237]{Gr2}. Therefore $v\in H^s(\Omega )$. We complete the proof by applying Rellich identity \eqref{1.5}.
\end{proof}

Lemma \ref{lemma1.1} at hand, we can mimic the proof of \cite[Theorem 7.6.1, page 252]{TW} in order to obtain the following theorem:

\begin{theorem}\label{theorem1.4}
We assume that $a\geq \delta$  on $\Gamma_1$, for some $\delta >0$. There exist $M\geq 1$ and $\omega >0$, depending only on $\delta$, so that
\[
\|e^{tA_a}(v,w)\|_{V\times L^2(\Omega )}\leq Me^{-\omega t}\|(v,w)\|_{V\times L^2(\Omega )},\;\; (v,w)\in D(A_a),\; t\geq 0.
\]
\end{theorem}

An immediate consequence of Theorem \ref{theorem1.4} is the following observability inequality.

\begin{corollary}\label{corollary1.1}
We fix $0<\delta _0 <\delta _1$. Then there exist $\tau _0>0$ and $\kappa$, depending only on $\delta _0$ and $\delta_1$ so that for any $\tau \geq \tau _0$ and $a\in \mathscr{A}$ satisfying  $\delta _0 \le a \le \delta _1$  on $\Gamma _1$,
\[
\| (u^0,u^1)\|_{V\times L^2(\Omega )}\leq \kappa \|C_a(u^0,u^1)\|_{L^2(\Sigma  _1)}.
\]
Moreover, $C_a$ is admissible for $e^{tA_a}$ and $(C_a, A_a)$ is exactly observable.
\end{corollary}
We omit the proof of this corollary. It is quite similar to that of \cite[Corollary 7.6.5, page 256]{TW}.


\section{The inverse problem}

\subsection{An abstract framework for the inverse source problem}

In the present subsection we consider an inverse source problem for an abstract evolution equation. The result of this subsection is the main ingredient in the proof of Theorem \ref{theorem1.1}.

\smallskip
Let $H$ be a Hilbert space  and $A :D(A)  \subset H \rightarrow H$ be the generator of continuous semigroup $(T(t))$. An operator  $C \in \mathscr{B}(D(A),Y)$,  $Y$ is a Hilbert space which is identified with its dual space, is called an admissible observation  for $(T(t))$ if for some (and hence for all) $\tau >0$, the operator $\Psi \in \mathscr{B}(D(A),L^2((0,\tau ),Y))$ given by
\[
(\Psi  x)(t)=CT(t)x,\;\; t\in [0,\tau ],\;\; x\in D(A),
\]
has a bounded extension to $H$.

\smallskip
We  introduce the notion of exact observability for the system
\begin{align}\label{2.1}
&z'(t)=Az(t),\;\; z(0)=x,
\\
&y(t)=Cz(t),\label{2.2}
\end{align}
where $C$ is an admissible observation  for $T(t)$. Following the usual definition, the pair $(A,C)$ is said exactly observable at time $\tau >0$ if there is a constant $\kappa $ such that the solution $(z,y)$ of \eqref{2.1} and \eqref{2.2} satisfies 
\[
\int_0^\tau \|y(t)\|_Y^2dt\geq \kappa ^2 \|x\|_H^2,\;\; x\in D(A).
\]
Or equivalently
\begin{equation}\label{2.3}
\int_0^\tau \|(\Psi  x)(t)\|_Y^2dt\geq \kappa ^2 \|x\|_H^2,\;\; x\in D(A).
\end{equation}

Let $\lambda \in H^1((0,\tau ))$ such that $\lambda (0)\ne 0$. We consider the Cauchy problem
\begin{equation}\label{2.4}
z'(t)=Az(t)+ \lambda (t)x,\;\; z(0)=0
\end{equation}
and we set
\begin{equation}\label{2.5}
y(t)=Cz(t),\;\; t\in [0,\tau ].
\end{equation}

We fix $\beta$ in the resolvent set of $A$. Let $H_1$ be the space $D(A)$ equipped with the norm $\|x\|_1=\|(\beta -A)x\|$ and denote by $H_{-1}$  the completion of $H$ with respect to the norm $\|x\|_{-1}=\| (\beta -A)^{-1}x\|$. As it is observed in \cite[Proposition 4.2, page 1644]{tucsnak} and its proof,  when $x\in H_{-1}$ (which is the dual space of $H_1$ with respect to the pivot space $H$) and $\lambda \in H^1((0,T))$, then, according to the classical extrapolation theory of semigroups,  the Cauchy problem \eqref{2.4} has a unique solution $z\in C([0,\tau ];H)$. Additionally $y$ given in \eqref{2.5} belongs to $L^2((0,\tau ) ,Y)$.

\smallskip
When $x\in H$, we have by Duhamel's formula
\begin{equation}\label{2.6}
y(t)=\int_0^t  \lambda (t-s)CT(s)xds=\int_0^t \lambda (t-s)(\Psi  x)(s)ds.
\end{equation}

Let
\[
H^1_\ell ((0,\tau), Y) = \left\{u \in H^1((0,\tau), Y); \; u(0) = 0 \right\}.
\]

We define the operator $S :L^2((0,\tau), Y)\longrightarrow H^1_\ell ((0,\tau ) ,Y)$ by
\begin{equation}\label{2.7}
(S h)(t)=\int_0^t \lambda (t-s)h(s)ds.
\end{equation}

If $E  =S \Psi$, then \eqref{2.6} takes the form
\[
y(t)=(E x)(t).
\]

Let $\mathcal{Z}=(\beta -A^\ast)^{-1}(X+C^\ast Y)$.

\begin{theorem}\label{theorem2.1} 
We assume that $(A,C)$ is exactly observable at time $\tau$. Then 
\\
(i) $E$ is one-to-one from $H$ onto $H^1_\ell ((0,\tau), Y)$.
\\
(ii) $E$ can be extended to an isomorphism, denoted by $\widetilde{E}$, from $\mathcal{Z}'$ onto $L^2((0,\tau );Y)$.
\\
(iii)
There exists a constant $\widetilde{\kappa}$, independent on $\lambda$, so that
\begin{equation}\label{2.8}
 \|x\|_{\mathcal{Z}'}\leq \widetilde{\kappa}|\lambda (0)|e^{\frac{\|\lambda '\|^2_{L^2((0,\tau ))}}{|\lambda (0)|^2}\tau}\|\widetilde{E}x\|_{L^2 ((0,\tau), Y)}.
\end{equation}
\end{theorem}

\begin{proof}
(i) and (ii) are contained in \cite[Theorem 4.3, page 1645]{tucsnak}. We need only to prove (iii). To do this, we start by observing that 
\[
S^\ast: L^2((0,\tau ),Y)\rightarrow H_r^1((0,\tau );Y)=\left\{u \in H^1((0,\tau), Y); \; u(\tau ) = 0 \right\},
\] 
the adjoint of $S$, is given by
\[
S^\ast h(t)=\int_t^\tau \lambda (s-t)h(s)ds,\;\; h\in H_r^1((0,\tau );Y).
\]

We fix $h\in H_r^1((0,\tau );Y)$ and we set $k=S^\ast h$. Then
\[
k'(t)= \lambda (0)h(t)-\int_t^\tau  \lambda '(s-t)h(s)ds.
\]
Hence
\begin{align*}
\left[|\lambda (0)| \|h(t)\|\right]^2&\le \left( \int_t^\tau \frac{|\lambda '(s-t)|}{|\lambda (0)|}[ |\lambda (0)|\|h(s)\|]ds +\|k'(t)\|\right)^2
\\
&\le 2\left( \int_t^\tau \frac{|\lambda '(s-t)|}{|\lambda (0)|}[ |\lambda (0)|\|h(s)\|]ds\right)^2 +2\|k'(t)\|^2
\\
&\le 2\frac{\|\lambda '\|_{L^2((0,\tau ))}^2}{|\lambda (0)|^2}\int_0^t [|\lambda (0)|\|h(s)\|]^2ds+2\|k'(t)\|^2.
\end{align*}
The last estimate is obtained  by applying Cauchy-Schwarz's inequality.

\smallskip
A simple application of Gronwall's lemma entails
\[
[|\lambda (0)|\|h(t)\|]^2\le 2e^{2\frac{\|\lambda '\|_{L^2((0,\tau ))}^2}{|\lambda (0)|^2}\tau }\|k'(t)\|^2.
\]
Therefore,
\[
\|h\|_{L^2((0,\tau) ;Y)}\leq \frac{\sqrt{2}}{|\lambda (0)|}e^{\frac{\|\lambda '\|_{L^2((0,\tau ))}^2}{|\lambda (0)|^2}\tau }\|k'\|_{L^2((0,\tau) ;Y)}.
\]
This inequality yields
\begin{equation}\label{2.8.1}
\|h\|_{L^2((0,\tau) ;Y)}\le \frac{\sqrt{2}}{|\lambda (0)|}e^{\frac{\|\lambda '\|_{L^2((0,\tau ))}^2}{|\lambda (0)|^2}\tau }\|S^\ast h\|_{H^1_r((0,\tau) ;Y)}.
\end{equation}
The adjoint of $S^\ast$, acting as a bounded operator from $[H_r((0,1);Y)]'$ into $L^2((0,\tau );Y)$, gives an extension of $S$. We denote by $\widetilde{S}$ this operator. By \cite[Proposition 4.1, page 1644]{tucsnak} $S^\ast$ defines an isomorphism from $[H_r((0,1);Y)]'$ onto $L^2((0,\tau );Y)$. In light of the fact that
\[
\|\widetilde{S}\|_{\mathscr{B}([H_r((0,1);Y)]';L^2((0,\tau );Y))}=\|S^\ast \|_{\mathscr{B}(L^2((0,\tau );Y);H_r((0,1);Y))},
\]
\eqref{2.8.1} implies
\begin{equation}\label{2.8.1b}
\frac{|\lambda (0)|}{\sqrt{2}}e^{-\frac{\|\lambda '\|_{L^2((0,\tau ))}^2}{|\lambda (0)|^2}\tau }\le \|\widetilde{S}\|_{\mathscr{B}([H_r((0,1);Y)]';L^2((0,\tau );Y))}.
\end{equation}
On the other hand, according to \cite[Proposition 2.13, page 1641]{tucsnak}, $\Psi $ possesses a unique bounded extension, denoted by $\widetilde{\Psi}$ from $\mathcal{Z}'$ into $[H_r((0,1);Y)]'$ and there exists a constant $c>0$ so that
\begin{equation}\label{2.8.2}
\|\widetilde{\Psi}\|_{\mathcal{B}(\mathcal{Z}';[H_r((0,1);Y)]')}\geq c.
\end{equation}
Consequently, $\widetilde{E}=\widetilde{S}\widetilde{\Psi}$ gives a unique extension of $E$ to an isomorphism from $\mathcal{Z}'$ onto $L^2((0,\tau );Y)$. 

\smallskip
We end up the proof by noting that \eqref{2.8} is a consequence of \eqref{2.8.1} and \eqref{2.8.2}.
\end{proof}

\subsection{An inverse source problem for an IBVP for the wave equation}

In the present subsection we are going to apply the result of the preceding subsection to $H=V\times L^2(\Omega )$, $H_1=D(A_a)$ equipped with its graph norm and $Y=L^2(\Gamma _1)$.

\smallskip
We consider the  the IBVP
\begin{equation}\label{2.9}
\left\{
\begin{array}{lll}
 \partial _t^2 u - \Delta u = \lambda (t)w \;\; &\mbox{in}\;   Q, 
 \\
u = 0 &\mbox{on}\;  \Sigma _0, 
\\
\partial _\nu u +a\partial _tu=0 &\mbox{on}\;  \Sigma _1, 
\\
u(\cdot ,0) =0,\; \partial_t u (\cdot ,0) =0.
\end{array}
\right.
\end{equation}

Let $(0,w)\in H_{-1}$ and $\lambda \in H^1((0,\tau ))$. From the comments in the preceding subsection, \eqref{2.9} has a unique solution $u_w$ so that $(u_w,\partial _tu_w) \in C([0,\tau ]; V\times L^2(\Omega ))$ and $\partial_\nu  u_w{_{|\Gamma _1}}\in L^2(\Sigma _1 )$.

\smallskip
We consider the inverse problem consisting in the determination of $w$, so that $(0,w)\in H_{-1}$, appearing in the IBVP \eqref{2.9} from the boundary measurement $\partial_\nu  u_w{_{|\Sigma _1}}$. Here the function $\lambda$ is assumed to be known.

\smallskip
Taking into account that $\{0\}\times V' \subset H_{-1}$, where $V'$ is the dual space of $V$, we obtain as a consequence of Corollary \ref{corollary2.1}:

\begin{proposition}\label{proposition2.1}
There exists a constant $C>0$ so that for any $\lambda \in H^1((0,\tau ))$ and $w\in V'$,
\begin{equation}\label{2.10}
 \|w\|_{V'}\leq C|\lambda (0)|e^{\frac{\|\lambda '\|^2_{L^2((0,\tau ))}}{|\lambda (0)|^2}\tau}\|\partial _\nu u\|_{L^2 (\Sigma _1)}.
\end{equation}
\end{proposition}

\subsection{Proof of Theorem \ref{theorem1.1}}

We start by observing that $u_a$ is also the unique solution of 
\begin{equation*}
\left\{
\begin{array}{lll}
 \int_\Omega u''(t)vdx=\int_\Omega \nabla u(t)\cdot \nabla vdx-\int_{\Gamma _1}au'(t)v,\;\; \textrm{for all}\; v\in V.
 \\
 u(0)=u^0,\;\; u'(0)=u^1.
\end{array}
\right.
\end{equation*}

Let $u=u_a-u_0$. Then $u$ is the solution of the following problem
\begin{equation}\label{2.11}
\left\{
\begin{array}{lll}
 \int_\Omega u''(t)vdx=\int_\Omega \nabla u(t)\cdot \nabla vdx-\int_{\Gamma _1}au'(t)v -\int_{\Gamma _1}au_0'(t)v,\;\; \textrm{for all}\; v\in V.
 \\
 u(0)=0,\;\; u'(0)=0.
\end{array}
\right.
\end{equation}

For $k$, $\ell \in \mathbb{Z}$, we set 
\begin{align*}
&\lambda_{k\ell}=[ (k+1/2)^2+(\ell +1/2)^2]\pi ^2
\\
&\phi_{k\ell}(x,y)=2\cos ((k+1/2)\pi x)\cos ((\ell +1/2)\pi y).
\end{align*}

We check in a straightforward manner that $u_0=\cos(\sqrt{\lambda_{k\ell}}t)\phi_{k\ell}$ when $(u^0,u^1)=(\phi_{k\ell},0)$. 

\smallskip
In the sequel $k$, $\ell$ are arbitrarily fixed. We set $\lambda (t)=\cos(\sqrt{\lambda_{k\ell}}t)$ and we define $w_a\in V'$ by
\[
w_a(v)=-\sqrt{\lambda_{k\ell}}\int_{\Gamma _1}a\phi_{k\ell}v.
\]

In that case \eqref{2.11} becomes
\begin{equation*}
\left\{
\begin{array}{lll}
 \int_\Omega u''(t)vdx=\int_\Omega \nabla u(t)\cdot \nabla vdx-\int_{\Gamma _1}au'(t)v +\lambda (t)w_a(v),\;\; \textrm{for all}\; v\in V.
 \\
 u(0)=0,\;\; u'(0)=0.
\end{array}
\right.
\end{equation*}
Consequently, $u$ is the solution of \eqref{2.9} with $w=w_a$. Applying Proposition \ref{proposition2.1}, we find
\begin{equation}\label{2.12}
\|w_a\|_{V'}\leq Ce^{\lambda_{k\ell}\tau ^2}\| \partial _\nu u\|_{L^2(\Sigma _1)}.
\end{equation}
But
\begin{equation}\label{2.13}
a_1(0)\left| \int_{\Gamma _1}(a\phi_{k\ell})^2d\sigma \right| =\frac{1}{\sqrt{\lambda_{k\ell}}}\left| w_a ((a_1\otimes a_2)\phi_{k\ell})\right|\leq \frac{1}{\sqrt{\lambda_{k\ell}}}\|w_a\|_{V'}\|(a_1\otimes a_2)\phi_{k\ell}\|_V,
\end{equation}
where we used $a_1(0)=a_2(0)$, and
\begin{equation}\label{2.14}
 \|(a_1\otimes a_2)\phi_{k\ell}\|_V \leq C_0\sqrt{\lambda _{kl}}\| a_1\otimes a_2\|_{H^1(\Omega )}.
 \end{equation}
 Here $C_0$ is a constant independent on $a$ and $\phi_{k\ell}$.

\smallskip
 We note $(a_1\otimes a_2)\phi_{k\ell}\in V$ even if  $a_1\otimes a_2\not\in V$.

\smallskip
Now a combination of \eqref{2.12}, \eqref{2.13} and \eqref{2.14} yields
\begin{equation*}
a_1(0)\left(\|a_1\phi _k\|_{L^2((0,1))}^2+\|a_2\phi _\ell\|_{L^2((0,1))}^2\right)\le C\| a_1\|_{H^1(0,1)}\| a_2\|_{H^1(0,1)}e^{\lambda_{k\ell}\tau ^2/2}\| \partial _\nu u\|_{L^2(\Sigma _1)},
\end{equation*}
where $\phi_k (s)=\sqrt{2}\cos ((k+1/2)\pi s)$. This and the fact that $m\leq a_j(0)$ and $\|a_j\|_{H^1((0,1))}\leq M$ imply
\begin{equation*}
\|a_1\phi _k\|_{L^2((0,1))}^2+\|a_2\phi _\ell\|_{L^2((0,1))}^2\le C\frac{M^2}{m}e^{\lambda_{k\ell}\tau ^2/2}\| \partial _\nu u\|_{L^2(\Sigma _1)}.
\end{equation*}
Hence, where $j=1$ or $2$,
\begin{equation*}
\|a_j\phi _k\|_{L^2((0,1))}^2\le C\frac{M^2}{m}e^{k^2\tau ^2\pi^2}\| \partial _\nu u\|_{L^2(\Sigma _1)}.
\end{equation*}

Let 
\[
a_j^k= \int_0^1a_j(x)\phi _k(x)dx,\;\; j=1,2.
\]

Since 
\[
|a_j^k|=\left| \int_0^1a_j(x)\phi _k(x)dx\right| \le \|a_j\phi _k\|_{L^1((0,1))}\leq \|a_j\phi _k\|_{L^2((0,1))},
\]
we get
\[
(a_j^k)^2\le C\frac{M^2}{m}e^{k^2\tau ^2\pi^2}\| \partial _\nu u\|_{L^2(\Sigma _1)}.
\]
On the other hand
\[
\| \partial _\nu u\|_{L^2(\Sigma _1)}=\|\Lambda _a (\phi_{kl})-\Lambda _0 (\phi_{kl})\|_{L^2(\Sigma )}\leq Ck^2\|\Lambda _a -\Lambda _0 \|.
\]
Hence
\begin{equation}\label{2.15}
(a_j^k)^2\le C\frac{M^2}{m}e^{k^2(\tau ^2\pi^2+1)}\|\Lambda _a -\Lambda _0 \|.
\end{equation}
Let $q=\frac{M^2}{m}$ and $\alpha =\tau ^2\pi^2+2$. We obtain in a straightforward manner from \eqref{2.15}
\[
\sum_{|k|\leq N}(a_j^k)^2\le Cqe^{\alpha N^2}\|\Lambda _a -\Lambda _0 \|.
\]
Consequently,
\begin{align*}
\|a_j\|_{L^2((0,1))}^2 &\le \sum_{|k|\leq N}(a_j^k)^2+\frac{1}{N^2}\sum_{|k|> N}k^2(a_j^k)^2
\\
&\le C\left(qe^{\alpha N^2}\|\Lambda _a -\Lambda _0 \|+ \frac{\|a_j\|_{H^1((0,1))}^2}{N^2}\right)
\\
&\le C\left(qe^{\alpha N^2}\|\Lambda _a -\Lambda _0 \|+ \frac{M^2}{N^2}\right)
\\
&
\le CM^2\left( \frac{1}{m}e^{\alpha N^2}\|\Lambda _a -\Lambda _0 \|+ \frac{1}{N^2}\right).
\end{align*}
That is 
\begin{equation}\label{2.16}
\|a_j\|_{L^2((0,1))}^2\leq CM^2\left( \frac{1}{m}e^{\alpha N^2}\|\Lambda _a -\Lambda _0 \|+ \frac{1}{N^2}\right).
\end{equation}

Assume that $\|\Lambda _a -\Lambda _0 \|\le \delta =me^{-\alpha}$. Let then $N_0\geq 1$ be the greatest integer so that
\[
\frac{C}{m}e^{\alpha N_0^2}\|\Lambda _a -\Lambda _0 \|\leq \frac{1}{N_0^2}.
\]
Using
\[
\frac{1}{m}e^{\alpha (N_0+1)^2}\|\Lambda _a -\Lambda _0 \|\le \frac{1}{(N_0+1)^2},
\]
we find
\[
(2N_0)^2\geq (N_0+1)^2\ge \frac{1}{\alpha +1}\ln \left( \frac{m}{\|\Lambda _a -\Lambda _0 \|}\right).
\]
This estimate in \eqref{2.16} with $N=N_0$ gives
\begin{equation}\label{2.17}
\|a_j\|_{L^2((0,1))} \le 2C\sqrt{\alpha +1}M\left| \ln \left(m^{-1}\|\Lambda _a -\Lambda _0 \|\right)\right|^{-1/2}.
\end{equation}

When $\|\Lambda _a -\Lambda _0 \|\ge \delta$, we have
\begin{equation}\label{2.18}
\|a_j\|_{L^2((0,1))} \le \frac{M}{\delta}\|\Lambda _a -\Lambda _0 \|.
\end{equation}

In light of \eqref{2.17} and \eqref{2.18}, we find a constants $c >0$, that can depend only on $\tau$,  so that
\[
 \|a_j\|_{L^2((0,1))} \le cM\left(\left| \ln \left(m^{-1}\|\Lambda _a -\Lambda _0 \|\right)\right|^{-1/2}+m^{-1}\|\Lambda _a -\Lambda _0 \|\right).
 \]

\appendix
\section{}\label{appendixA}
 
 \medskip
 We prove the following lemma
 \begin{lemma}\label{lemmaA1}
 Let $1/2<\alpha \leq1$ and $a\in C^\alpha ([0,1])$. Then the mapping $f \mapsto af$ defines a bounded operator on $H^{1/2}((0,1))$.
 \end{lemma}
 
 \begin{proof}
 We recall that $H^{1/2}((0,1))$ consists in functions $f\in L^2((0,1))$ with finite norm
 \[
 \|f\|_{H^{1/2}((0,1))}=\left( \|f\|^2_{L^2((0,1))}+\int_0^1\int_0^1\frac{|f(x)-f(y)|^2}{|x-y|^2}dxdy\right)^{1/2}.
\]
Let $a\in C^\alpha( [0,1])$. We have
\[
\frac{|a(x)f(x)-a(y)f(y)|^2}{|x-y|^2} \leq \|a\|_{L^\infty (0,1)}^2\frac{|f(x)-f(y)|^2}{|x-y|^2}+|f(y)|^2\frac{[a]_\alpha^2}{|x-y|^{2(1-\alpha)}},
\]
where
\[
[a]_\alpha =\sup \{|a(x)-a(y)||x-y|^{-\alpha};\; x,y\in [0,1],\; x\neq y\}.
\]
Using that $1/2<\alpha \leq1$, we find that $x\rightarrow |x-y|^{-2(1-\alpha)}\in L^1((0,1))$, $y\in [0,1]$, and
\[
\int_0^1\frac{dx}{|x-y|^{2(1-\alpha)}}\leq \frac{1}{2\alpha -1},\;\; y\in [0,1].
\]
Hence $af\in H^{1/2}((0,1))$ with
\[
\|af\|_{H^{1/2}((0,1))}\leq \frac{1}{2\alpha -1}\|a\|_{C^\alpha ([0,1])}\|f\|_{H^{1/2}((0,1))}.
\]
Here
\[
\|a\|_{C^\alpha ([0,1])}=\|a\|_{L^\infty ((0,1))}+[a]_\alpha .
\]
 \end{proof}

\section{}\label{appendixB}

We give the proof of the following lemma

\begin{lemma}\label{lemmaB1}
Let $a\in \mathscr{A}$ and $A_a$ be the unbounded operator defined on $V\times L^2(\Omega )$ by
\[
A_a= (w,\Delta v),\quad D(A_a)=\{ (v,w)\in V\times V;\; \Delta v\in L^2(\Omega )\; \textrm{and}\; \partial _\nu v=-aw\; \textrm{on}\; \Gamma _1\}.
\]
Then $A_a$ is m-dissipative.
\end{lemma}

\begin{proof}
Let $\langle \cdot ,\cdot \rangle$ be scalar product in $V\times L^2(\Omega )$. That is
\[
\langle (v_1,w_1),(v_2,w_2)\rangle =\int_\Omega \nabla v_1\cdot \nabla \overline{v_2}dx+\int_\Omega w_1\overline{w_2}dx,\;\; (v_j,w_j)\in V\times L^2(\Omega ),\; j=1,2.
\]
For $(v_1,w_1)\in D(A_a)$, we have
\begin{align}
\langle A_a(v_1,w_1),(v_1,w_1)\rangle &= \langle (w_1,\Delta v_1),(v_1,w_1)\rangle \label{B1}
\\
&=\int_\Omega \nabla w_1\cdot \nabla \overline{v_1}dx+\int_\Omega \Delta v_1\overline{w_1}dx \nonumber
\end{align}
Applying twice Green's formula, we get
\begin{align}
&\int_\Omega \nabla w_1\cdot \nabla \overline{v_1}dx=-\int_\Omega  w_1\Delta \overline{v_1}dx+\int_{\Gamma _1}w_1\partial _\nu  \overline{v_1}d\sigma ,\label{B2}
\\
&\int_\Omega \Delta v_1\overline{w_1}dx =- \int_\Omega \nabla v_1\cdot \nabla \overline{w_1}dx - \int_{\Gamma _1}aw_1\overline{w_1}d\sigma.\label{B3}
\end{align}
We take the sum side by side of identities \eqref{B2} and \eqref{B3}. Using that $\partial_\nu v_1=-aw_1$ on $\Gamma _1$ we obtain 
\begin{align*}
\int_\Omega \nabla w_1\cdot \nabla \overline{v_1}dx+\int_\Omega \Delta v_1\overline{w_1}dx &=-\int_\Omega  w_1\Delta \overline{v_1}dx- \int_\Omega \nabla v_1\cdot \nabla \overline{w_1}dx - 2 \, \int_{\Gamma_1} a \left|w_1\right|^2 d\sigma 
\\
&=-\langle (v_1,w_1),A_a(v_1,w_1)\rangle - 2 \, \int_{\Gamma_1} a \left|w_1\right|^2d\sigma.
\end{align*}
This and \eqref{B1} yield
\[
\Re \langle A_a(v_1,w_1),(v_2,w_2)\rangle = -  \, \int_{\Gamma_1} a \left|w_1\right|^2d\sigma \leq 0.
\]
In other words, $A_a$ is dissipative. 

\smallskip
We complete the proof by showing that $A_a$ is onto implying that $A_a$ is m-dissipative. To this end we are going to show that for each $(f,g)\in V\times L^2(\Omega )$, the problem 
\[
w=f, \quad -\Delta v=g.
\]
has a unique solution $(v,w)\in D(A_a)$.

\smallskip
In light of the fact $\psi \rightarrow \left(\int_\Omega |\nabla \psi |^2dx\right)^{1/2}$ defines an equivalent norm on $V$, we can apply Lax-milgram's lemma. We get that there exists a unique $v\in V$ satisfying
\[
\int_\Omega \nabla v\cdot \nabla \overline{\psi}dx=\int_\Omega g\overline{\psi}dx -\int_{\Gamma _1}aw\overline{\psi}d\sigma ,\;\; \psi \in V.
\]
From this identity, we deduce in a standard way that $-\Delta v=g$ and $\partial _\nu v=-aw$ on $\Gamma _1$. The proof is then complete
\end{proof}


\bigskip

\end{document}